\documentclass[12pt, reqno, draft]{article}

\usepackage[utf8x]{inputenc}
\usepackage[T2A]{fontenc}
\usepackage[russian, english]{babel}
\usepackage{amsmath}
\usepackage{amsfonts,amssymb, amsthm, amscd} 
\usepackage{hyperref}
\hypersetup{unicode=true,final=true,colorlinks=true}

\usepackage{authblk}

\theoremstyle{plain}

\newtheorem{definicion}{Definition}

\newtheorem{teorema}{Theorem}
\newtheorem{proposition}{Proposition}
\newtheorem{corolario}{Corollary}

\theoremstyle{remark}
\newtheorem{remark}{Remark}

\DeclareMathOperator*{\slim}{slim}

\newcommand{\Id}{{I}}
\newcommand{\Hilbert}{{\mathcal H}}
\newcommand{\Banach}{{X}}

\newcommand{\N}{{\mathbb N}}
\newcommand{\R}{{\mathbb R}}

\newcommand{\Lin}{{\mathcal{L} }}

\newcommand{\Interval}{\mathcal{I}}

\author[$\dagger$]{Nikita Artamonov\thanks{nikita.artamonov@gmail.com}}

\affil[$\dagger$]{MGIMO University}

\title{Solvability of an Operator Riccati Integral Equation in a
Reflexive Banach Space}

\begin{document}

\maketitle

\begin{abstract}
We show that if  $\Banach$ is a reflexive Banach space, then a nonautonomous
operator Riccati integral equation has a unique strongly continuous
self-adjoint nonnegative solution $P(t)\in\Lin(\Banach,\Banach^*)$
\end{abstract}

\section{Preliminaries}

It is well known
\cite{RepresentationAndControl, CurtainPritchard, Lasiecka2004, Gibson1979, PritchardSalamon},
that the solution of a linear-quadratic control problem on a finite interval
can be expressed via the solution of an operator Riccati (differential or
integral) equation considered in the space of operator functions.

Some results on the solvability of autonomous and nonautonomous Riccati
equations in operator functions ranging in the space $\Lin(\Hilbert)$
where $\Hilbert$ is a Hilbert space, were obtained in
\cite{RepresentationAndControl, Lasiecka2004, LasieckaTriggiani1992,
DaPratoIchikawa} and \cite{CurtainPritchard, Gibson1979}, respectively.

A triple $\Banach\hookrightarrow\Hilbert\hookrightarrow\Banach^*$ of
spaces with dense embeddings was considered in \cite{PritchardSalamon}
for a Hilbert space $\Banach$ and in \cite{Artamonov} for a reflexive
Banach space $\Banach$. In these papers, the solvability of an autonomous
Riccati equation in operator functions ranging in the spaces
$\Lin(\Banach^*,\Banach)$ and $\Lin(\Banach,\Banach^*)$, respectively,
was established. In the papers \cite{Artamonov, Yong}, the solvability of
the Riccati equation was used to prove the solvability of systems of
forward–backward evolution equations.

The present paper generalizes the above-mentioned results.
We prove that there exists a unique solution of the Riccati integral equation
for strongly continuous operator functions ranging in the space
$\Lin(\Banach,\Banach^*)$, where $\Banach$ is an arbitrary reflexive
Banach space. It is important to note that, in contrast to the papers
\cite{Artamonov, PritchardSalamon}, we do not assume an embedding between
the space $\Banach$ and the dual space.

\textbf{1.1} By $\Lin(\Banach_1,\Banach_2)$ we denote the normed space of
continuous linear operators from a Banach space $\Banach_1$ to a Banach space
$\Banach_2$. Just as in \cite[Part IV]{RepresentationAndControl}, we introduce
the following spaces of operator functions. By
$C_u([a,b];\Lin(\Banach_1,\Banach_2))$
we denote the Banach space of strongly continuous operator functions on
the interval $[a,b]$ ranging in $\Lin(\Banach_1,\Banach_2)$ with the norm
\[
	\|P\|_u=\sup_{t\in[a,b]}\|P(t)\|_{\Lin(\Banach_1,\Banach_2)}.
\]
and by $C_s([a,b];\Lin(\Banach_1,\Banach_2))$ we denote the topological
space of strongly continuous operator functions on $[a,b]$ ranging in
$\Lin(\Banach_1,\Banach_2)$ with the topology of uniform strong convergence.
By definition, $P\in C_s([a,b];\Lin(\Banach_1,\Banach_2))$ if and only if
the vector function $(Px)(t)=P(t)x$ belongs to the Banach space
$C([a,b];\Banach_2)$ for each $x\in\Banach_1$. If
$\Banach_1=\Banach_2=Y$, then we write $C_s([a,b]l;\Lin(Y))$ instead of
$C_s([a,b];\Lin(Y,Y))$. Note that if
$P\in C_s([a,b];\Lin(\Banach_1,\Banach_2))$, then the function
$\|P(\cdot)\|_{\Lin(\Banach_1,\Banach_2)}$ is measurable and bounded and
the function $\|P(\cdot)x\|_{\Banach_2}\in C[a,b]$ for each $x\in\Banach_1$.
By definition, a sequence $\{P_k\}_{k=1}^{+\infty}$ converges to $P$
in the space $C_s([a,b];\Lin(\Banach_1,\Banach_2))$ if and only if
the sequence of vector functions $P_kx$ converges to the vector function
$Px$ uniformly on $[a,b]$ (i.e., $P_kx$ converges to $Px$ in space
$C([a,b];\Banach_2)$) for each $x\in\Banach_1$. A straightforward verification
shows that if $P\in C_s([a,b];\Lin(\Banach_1,\Banach_2))$ and
$Q\in C_s([a,b];\Lin(\Banach_2,\Banach_3))$, then
$QP\in C_s([a,b];\Lin(\Banach_1,\Banach_3))$.

The topological space of strongly continuously differentiable operator functions
$C_s^1([a,b];\Lin(\Banach_1,\Banach_2))$ with the topology of uniform
strong convergence is defined in a similar way. By definition,
$P\in C_s^1([a,b];\Lin(\Banach_1,\Banach_2))$ if and only if the vector function
$Px$ belongs to space $C^1([a,b];\Banach_2)$ for each $x\in\Banach_1$.

Throughout the paper, $\slim$ stands for the limit in the strong operator
topology; for convenience, we denote the interval $\Interval=[0,T]$.

\textbf{1.2} Let $\Banach_{1}$ and $\Banach_{2}$ be Banach spaces,
and let operators $A_{1}\in\Lin(\Banach_{1})$, $A_{2}\in\Lin(\Banach_{2})$,
$C,G\in \Lin(\Banach_1,\Banach_2)$ and $B\in\Lin(\Banach_2,\Banach_1)$ be given.
Since the operators $A_{1}$ and $A_{2}$ are bounded, it follows that they are
the generators of  $C_0$-groups $e^{tA_i}\in\Lin(\Banach_{1,2})$
($t\in\R$), $i=1,2$.

In the collection of spaces $(\Banach_1,\Banach_2)$, consider the autonomous
backward (in time) Riccati differential equation
\[
	P'(t)=-C+A_2P(t)+P(t)A_1-P(t)BP(t)\quad P(T)=G.
\]
on interval $\Interval$. A straightforward verification shows that if
operator function $P\in C^1_s(\Interval;\Lin(\Banach_1,\Banach_2))$ is a solution
of this equation, then the operator function $P$ satisfies
the integral equation
\[
	P(t)=e^{(t-T)A_2}Ge^{(t-T)A_1}+
	\int_t^T e^{(t-r)A_2}(C-P(r)BP(r))e^{(t-r)A_1}dr,
\]
where integral is understood in the strong sense. This equation can be called
an \emph{autonomous Riccati integral equation}.

Let $Y$ be a Banach space.
\begin{definicion}
An operator function $\{U_{t,s}\}_{0\leq s\leq t\leq T}\subset\Lin(Y)$
is called \emph{forward (in time) evolution family} in $\Lin(Y)$ if it has
the following properties:
\begin{enumerate}
	\item The relation $U_{s,s}=\Id_Y$ holds for each $s\in[0,T]$;
	\item The relation $U_{t,s}=U_{t,r}U_{r,s}$ holds for each
	$0\leq s\leq r\leq t\leq T$.
\end{enumerate}
\end{definicion}
\begin{definicion}
An operator function $\{V_{t,s}\}_{0\leq t\leq s\leq T}\subset\Lin(Y)$
is called a \emph{backward (in time) evolution family} in $\Lin(Y)$
 if it has the following properties:
\begin{enumerate}
	\item The relation $V_{s,s}=\Id_Y$ holds for each $s\in[0,T]$;
	\item The relation $V_{t,s}=V_{t,r}V_{r,s}$ holds for each
	$0\leq t\leq r\leq s\leq T$.
\end{enumerate}
\end{definicion}
\begin{remark}
It readily follows from these definitions that if $U_{t,s}$ is a forward
evolution family in  $\Lin(Y)$, then $V_{\tau,\sigma}=U^*_{\sigma,\tau}$ is
a backward evolution family in $\Lin(Y^*)$.
\end{remark}
\begin{definicion}
A (forward or backward) evolution family $U_{t,s}$ is said to be
\emph{strongly continuous} if it is strongly continuous in $t$ and $s$
separately, i.e., strongly continuous in $t$ for each $s$ and in $s$
for each $t$.
\end{definicion}
\begin{remark}
A strongly continuous (forward or backward) evolution family is not necessarily
jointly strongly continuous in $(t,s)$. Moreover, it may not be even uniformly
bounded in the operator norm \cite[Appendix B]{Gibson1979}
\end{remark}
\begin{remark}
In what follows, we conveniently use arrows to indicate forward and
backward evolution families; namely, we write $\overleftarrow{U}_{t,s}$ and $\overrightarrow{U}_{t,s}$ respectively.
\end{remark}

\begin{definicion}
Let $\overleftarrow{U}_{t,s}$ be a strongly continuous forward evolution family
in $\Lin(X_1)$, let $\overrightarrow{U}_{t,s}$ be a strongly continuous
backward evolution family in $\Lin(X_2)$ and assume that
\[
	C\in C_s(\Interval;\Lin(\Banach_1,\Banach_2)),\quad
	B\in C_s(\Interval;\Lin(\Banach_2,\Banach_1)),\quad
	G\in\Lin(\Banach_1,\Banach_2).
\]
The integral equation
\begin{equation}\label{IntegralRiccatiEq}
	P(t)=\overrightarrow{U}_{t,T}G\overleftarrow{U}_{T,t}+
	\int_t^T \overrightarrow{U}_{t,r}\{C(r)-P(r)B(r)P(r)\}\overleftarrow{U}_{r,t}dr
\end{equation}
for an operator function $P\in C_s(\Interval;\Lin(\Banach_1,\Banach_2))$
will be called the \emph{backward (in time) Riccati integral equation}
with the condition $P(T)=G$ in the collection of spaces $(\Banach_1,\Banach_2)$.
The integral is understood in the strong sense.
\end{definicion}
\begin{remark}
It follows from Definition \ref{IntegralRiccatiEq} and the semigroup
property of evolution families that if
$P\in C_s(\Interval;\Lin(\Banach_1,\Banach_2))$ is a solution of the Riccati
integral equation \eqref{IntegralRiccatiEq}, then the relation
\[
	P(t)=\overrightarrow{U}_{t,\tau}P(\tau)\overleftarrow{U}_{\tau,t}+
	\int_t^\tau \overrightarrow{U}_{t,r}\{C(r)-P(r)B(r)P(r)\}\overleftarrow{U}_{r,t}dr
\]
holds for all $0\leq t\leq \tau\leq T$.
\end{remark}

\textbf{1.3} The following result for a Banach space $Y$ is well known
\cite[Theorem 9.19]{EngelNagel}
\begin{teorema}
Let $\overleftarrow{U}_{t,s}$ be a strongly continuous uniformly bounded
forward evolution family in $\Lin(Y)$, and let $Q\in C_s(\Interval;\Lin(Y))$.
Then there exists a unique strongly continuous uniformly
bounded forward evolution family $\overleftarrow{\Psi}_{t,s}$ in $\Lin(Y)$,
satisfying the equations  ($0\leq s\leq t\leq T$)
\begin{align*}
	\overleftarrow{\Psi}_{t,s} &= \overleftarrow{U}_{t,s}+
	\int_s^t \overleftarrow{U}_{t,r}Q(r)\overleftarrow{\Psi}_{r,s}dr \\
	\overleftarrow{\Psi}_{t,s} &= \overleftarrow{U}_{t,s}+
	\int_s^t \overleftarrow{\Psi}_{t,r}Q(r)\overleftarrow{U}_{r,s}dr
\end{align*}
\end{teorema}
A similar result is true for strongly continuous uniformly bounded backward
evolution families.

Let us show that the family $\overleftarrow{\Psi}_{t,s}$ continuously depends
on the operator function $Q$.
\begin{proposition}\label{ContinuousDepOfPhiEvolutionFamily}
Let $\overleftarrow{U}_{t,s}$ be a uniformly bounded strongly continuous
forward evolution family in $\Lin(Y)$, and let a sequence
$\{Q_n\}_{n=1}^{+\infty}$ of operator functions converge to
$Q$ in the space $C_s(\Interval;\Lin(Y))$.
Further, let strongly continuous forward evolution families
$\overleftarrow{\Psi}^{(n)}_{t,s}$ and
$\overleftarrow{\Psi}_{t,s}$ in $\Lin(Y)$ be solutions of
the equations ($0\leq s\leq t\leq T$)
\begin{align*}
	\overleftarrow{\Psi}^{(n)}_{t,s} &= \overleftarrow{U}_{t,s}+
	\int_s^t \overleftarrow{U}_{t,r}Q_n(r)\overleftarrow{\Psi}^{(n)}_{r,s}dr\\
	\overleftarrow{\Psi}_{t,s} &= \overleftarrow{U}_{t,s}+
	\int_s^t \overleftarrow{U}_{t,r}Q(r)\overleftarrow{\Psi}_{r,s}dr.
\end{align*}
Then for each $s\in\Interval$ here exists a limit
$\slim_{n\to+\infty}\overleftarrow{\Psi}^{(n)}_{t,s}=\overleftarrow{\Psi}_{t,s}$
uniformly with respect to $t\in[s,T]$.
\end{proposition}
\begin{proof}
Let $\|\overleftarrow{U}_{t,s}\|\leq M_U$.
By the uniform boundedness principle, the inequalities
$\|Q_n(t)\|,\|Q(t)\|\leq M_Q$ hold with some constant $M_Q$.
The definition of evolution families implies the relation
\begin{multline*}
	\overleftarrow{\Psi}^{(n)}_{t,s}-\overleftarrow{\Psi}_{t,s}=
	\int_s^t\overleftarrow{U}_{t,r}\left[Q_n(r)\overleftarrow{\Psi}^{(n)}_{r,s}-
	Q(r)\overleftarrow{\Psi}_{r,s}\right]dr=\\
	\int_s^t\overleftarrow{U}_{t,r}[Q_n(r)-Q(r)]\overleftarrow{\Psi}_{r,s}dr+
	\int_s^r\overleftarrow{U}_{t,r}Q_n(r)\left[\overleftarrow{\Psi}^{(n)}_{r,s}-\overleftarrow{\Psi}_{r,s}\right]dr.
\end{multline*}
Hence for an arbitrary $x\in Y$ we obtain
\begin{multline*}
	\|(\overleftarrow{\Psi}^{(n)}_{t,s}-\overleftarrow{\Psi}_{t,s})x\|\leq
	M_U\int_s^t\left\|[Q_n(r)-Q(r)]\overleftarrow{\Psi}_{r,s}x\right\|dr+\\
	\int_s^tM_U M_Q\left\|(\overleftarrow{\Psi}^{(n)}_{r,s}-\overleftarrow{\Psi}_{r,s})x\right\|dr.
\end{multline*}
Since the first integral term is monotone nondecreasing with respect to $t$,
it follows from the Gronwall inequality that
\[
	\left\|(\overleftarrow{\Psi}^{(n)}_{t,s}-\overleftarrow{\Psi}_{t,s})x\right\|\leq
	M_U\exp(M_U M_Q(t-s))\int_s^t\|[Q_n(r)-Q(r)]\overleftarrow{\Psi}_{r,s}x\|dr.
\]
This, together with the uniform strong convergence of the sequence $Q_n$ to $Q$
and the uniform boundedness of the strongly continuous family
$\overleftarrow{\Psi}_{t,s}$, implies that $Q_n(r)\overleftarrow{\Psi}_{r,s}$
strongly converges to $Q(r)\overleftarrow{\Psi}_{r,s}$ uniformly with respect to
$r\in[s,T]$. The proof of the proposition is complete.
\end{proof}
\begin{corolario}\label{ContinuousDepOfPsiEvolutionFamily}
Let $\overrightarrow{U}_{t,s}$ be a uniformly bounded backward evolution family
in $\Lin(Y)$, and let a sequence $\{Q_n\}_{n=1}^{+\infty}$ converge to $Q$
in the space $C_s(\Interval;\Lin(Y))$. Further, let strongly continuous
backward evolution families $\overrightarrow{\Psi}^{(n)}_{t,s}$ and
$\overrightarrow{\Psi}_{t,s}$be solutions of the equations
($0\leq t\leq s\leq T$)
\begin{align*}
	\overrightarrow{\Psi}^{(n)}_{t,s} &= \overrightarrow{U}_{t,s}+
	\int_t^s \overrightarrow{U}_{t,r}Q_n(r)\overrightarrow{\Psi}^{(n)}_{r,s}dr\\
	\overrightarrow{\Psi}_{t,s} &= \overrightarrow{U}_{t,s}+
	\int_t^s \overrightarrow{U}_{t,r}Q(r)\overrightarrow{\Psi}_{r,s}dr.
\end{align*}
Then for each $t\in\Interval$ there exists a limit
$\slim_{n\to+\infty}\overrightarrow{\Psi}^{(n)}_{t,s}=\overrightarrow{\Psi}_{t,s}$
uniformly with respect $s\in[t,T]$.
\end{corolario}

\section{Representation of the solution of the Riccati equation}

\textbf{2.1.} We will need the following results on the form of solutions of
integral equations. Just as before, let $\Banach_1$ and $\Banach_2$ be
be Banach spaces, and let the following conditions be satisfied:
\begin{enumerate}
	\item\label{ForwardEvolutionFamilyU} $\overleftarrow{U}_{t,s}$ is a strongly
	continuous uniformly bounded forward evolution family in
	$\Lin(\Banach_1)$;
	\item\label{BackwardEvolutionFamilyU}   $\overrightarrow{U}_{t,s}$
	is a strongly continuous uniformly bounded backward evolution family in
	$\Lin(\Banach_2)$;
	\item\label{FunctionsQ} The operator functions $Q_{12}$, $Q_1$ and $Q_2$
	satisfy the inclusions
	\[
		Q_{12}\in C_s(\Interval; \Lin(\Banach_1,\Banach_2)),\quad
		Q_1\in C_s(\Interval; \Lin(\Banach_1)),\quad
		Q_2\in C_s(\Interval; \Lin(\Banach_2)).
	\]
\end{enumerate}
\begin{proposition}\label{AuxIntergralEquationProposition1}
Let conditions \ref{ForwardEvolutionFamilyU}, \ref{BackwardEvolutionFamilyU}
and \ref{FunctionsQ} be satisfied, and let a strongly continuous uniformly
bounded forward evolution family
$\{\overleftarrow{\Omega}_{t,s}\}_{0\leq s\leq t\leq T}$ in $\Lin(\Banach_1)$
be the unique solution of the equation
\[ 
	\overleftarrow{\Omega}_{t,s}=\overleftarrow{U}_{t,s}-
	\int_s^t \overleftarrow{\Omega}_{t,r}Q_1(r)\overleftarrow{U}_{r,s}dr,\quad
	0\leq s\leq t\leq T.
\]
Then for an arbitrary $G\in\Lin(\Banach_1,\Banach_2)$ the equation
\begin{equation}\label{AuxIntergralEquation1}
	P(t)=\overrightarrow{U}_{t,T}G\overleftarrow{U}_{T,t}+
	\int_t^T\overrightarrow{U}_{t,r}[Q_{12}(r)-P(r)Q_1(r)]\overleftarrow{U}_{r,t}dr,
\end{equation}
for an operator function $P\in C_s(\Interval; \Lin(\Banach_1,\Banach_2))$
has the unique solution
\begin{equation}\label{AuxIntergralEquation1Solution}
	P(t)=\overrightarrow{U}_{t,T}G\overleftarrow{\Omega}_{T,t}+
	\int_t^T\overrightarrow{U}_{t,r}Q_{12}(r)\overleftarrow{\Omega}_{r,t}dr,\quad
	t\in\Interval.
\end{equation}
\end{proposition}
\begin{proof}
Let us substitute the operator function \eqref{AuxIntergralEquation1Solution}
into the right-hand side of Eq. \eqref{AuxIntergralEquation1}.
Taking into account the semigroup property of $\overrightarrow{U}_{t,s}$, and
the definition of the family $\overleftarrow{\Omega}_{t,s}$ and changing
the order of integration, we obtain
\begin{multline*}
	\overrightarrow{U}_{t,T}G\overleftarrow{U}_{T,t}+
	\int_t^T\overrightarrow{U}_{t,r}[Q_{12}(r)-P(r)Q_1(r)]\overleftarrow{U}_{r,t}dr=
	\overrightarrow{U}_{t,T}G\overleftarrow{U}_{T,t}+\\
	\int_t^T\overrightarrow{U}_{t,r}\left[Q_{12}(r)-\left[\overrightarrow{U}_{r,T}G\overleftarrow{\Omega}_{T,r}+
	\int_r^T \overrightarrow{U}_{r,s}Q_{12}(s)\overleftarrow{\Omega}_{s,r}ds\right]Q_1(r)\right]\overleftarrow{U}_{r,t}dr=\\
	\overrightarrow{U}_{t,T}G\left[\overleftarrow{U}_{T,t}-
	\int_t^T\overleftarrow{\Omega}_{T,r}Q_1(r)\overleftarrow{U}_{T,r}dr\right]+
	\int_t^T\overrightarrow{U}_{t,r}Q_{12}(r)\overleftarrow{U}_{r,t}dr-\\
	\int_t^T\left[\int_r^T \overrightarrow{U}_{t,s}Q_{12}(s)\overleftarrow{\Omega}_{s,r}Q_1(r)\overleftarrow{U}_{r,t}ds\right]dr=
	\overrightarrow{U}_{t,T}G\overleftarrow{\Omega}_{T,t}+\\
	\int_t^T\overrightarrow{U}_{t,s}Q_{12}(s)\overleftarrow{U}_{s,t}ds-
	\int_t^T\left[\int_t^s \overrightarrow{U}_{t,s}Q_{12}(s)\overleftarrow{\Omega}_{s,r}Q_1(r)\overleftarrow{U}_{r,t}dr\right]ds=\\
	\overrightarrow{U}_{t,T}G\overleftarrow{\Omega}_{T,t}+
	\int_t^T\overrightarrow{U}_{t,s}Q_{12}(s)\left[\overleftarrow{U}_{s,t}-
	\int_t^s\overleftarrow{\Omega}_{s,r}Q_1(r)\overleftarrow{U}_{r,t}dr\right]ds=\\
	\overrightarrow{U}_{t,T}G\overleftarrow{\Omega}_{T,t}+
	\int_t^T \overrightarrow{U}_{t,s}Q_{12}(s)\overleftarrow{\Omega}_{s,t}ds=P(t).
\end{multline*}
Thus, the operator function \eqref{AuxIntergralEquation1Solution}
satisfies Eq. \eqref{AuxIntergralEquation1}.

The uniqueness of the solution of Eq. \eqref{AuxIntergralEquation1}
follows from the fact that, for sufficiently small $\delta>0$
the mapping $F$ acting by the rule
\[
	F(P)(t)=\overrightarrow{U}_{t,\tau}G\overleftarrow{U}_{\tau,t}+
	\int_t^{\tau}\overrightarrow{U}_{t,r}[Q_{12}(r)-P(r)Q_1(r)]\overleftarrow{U}_{r,t}dr
\]
is a contraction on the space $C_u([\tau-\delta,\tau];\Lin(X_1,X_2))$ for all
$\delta\leq \tau\leq T$. The proof of the proposition is complete.
\end{proof}
In a similar way, one can prove the following proposition.
\begin{proposition}\label{AuxIntergralEquationProposition2}
Let conditions \ref{ForwardEvolutionFamilyU}, \ref{BackwardEvolutionFamilyU} and
\ref{FunctionsQ} be satisfied, and let a strongly continuous uniformly bounded
backward evolution family $\overrightarrow{\Omega}_{t,s}$ in  $\Lin(\Banach_2)$
be the unique solution of the equation
\[ 
	\overrightarrow{\Omega}_{t,s}=\overrightarrow{U}_{t,s}-
	\int_t^s \overrightarrow{U}_{t,r}Q_2(r)\overrightarrow{\Omega}_{r,s}dr,\quad
	0\leq t\leq s\leq T.
\]
Then for an arbitrary $G\in\Lin(\Banach_1,\Banach_2)$ the equation
\begin{equation}\label{AuxIntergralEquation2}
	P(t)=\overrightarrow{U}_{t,T}G\overleftarrow{U}_{T,t}+
	\int_t^T\overrightarrow{U}_{t,r}[Q_{12}(r)-Q_2(r)P(r)]\overleftarrow{U}_{r,t}dr
\end{equation}
for an operator function $P\in C_s(\Interval;\Lin(\Banach_1,\Banach_2))$
has the unique solution
\begin{equation}\label{AuxIntergralEquation2Solution}
	P(t)=\overrightarrow{\Omega}_{t,T}G \overleftarrow{U}_{T,t}+
	\int_t^T\overrightarrow{\Omega}_{t,r}Q_{12}(r)\overleftarrow{U}_{r,t}dr.
\end{equation}
\end{proposition}
\begin{proposition}\label{AuxIntergralEquationProposition3}
Let conditions \ref{ForwardEvolutionFamilyU}, \ref{BackwardEvolutionFamilyU} and
\ref{FunctionsQ} be satisfied, and let strongly continuous uniformly bounded
(forward and backward) evolution families $\overleftarrow{\Omega}_{t,s}$ and
$\overrightarrow{\Omega}_{t,s}$ be defined in the same way as in Propositions
\ref{AuxIntergralEquationProposition1} and \ref{AuxIntergralEquationProposition2}, respectively.
Then for an arbitrary $G\in\Lin(\Banach_1,\Banach_2)$ the equation
\begin{equation}\label{AuxIntergralEquation3}
	P(t)=\overrightarrow{U}_{t,T}G\overleftarrow{U}_{T,t}+
	\int_t^T\overrightarrow{U}_{t,r}[Q_{12}(r)-P(r)Q_1(r)-
  Q_2(r)P(r)]\overleftarrow{U}_{r,t}dr
\end{equation}
for an operator function $P\in C_s(\Interval; \Lin(\Banach_1,\Banach_2))$
has the unique solution
\begin{equation}\label{AuxIntergralEquation3Solution}
	P(t)=\overrightarrow{\Omega}_{t,T}G \overleftarrow{\Omega}_{T,t}+
	\int_t^T\overrightarrow{\Omega}_{t,r}Q_{12}(r)\overleftarrow{\Omega}_{r,t}dr.
\end{equation}
\end{proposition}
\begin{proof}
Set $\tilde{Q}_{12}(t)=Q_{12}(t)-P(t)Q_1(t)$. Then Eq.
\eqref{AuxIntergralEquation3} can be written in the form
\[
	P(t)=\overrightarrow{U}_{t,T}G\overleftarrow{U}_{T,t}+
	\int_t^T\overrightarrow{U}_{t,r}[\tilde{Q}_{12}(r)-Q_2(r)P(r)]\overleftarrow{U}_{r,t}dr.
\]
By Proposition \ref{AuxIntergralEquationProposition2}, its solution has the form
\[
	P(t)=\overrightarrow{\Omega}_{t,T}G \overleftarrow{U}_{T,t}+
	\int_t^T\overrightarrow{\Omega}_{t,r}\tilde{Q}_{12}(r)\overleftarrow{U}_{r,t}dr
\]
i.e., the operator function $P$ satisfies the equation
\[
	P(t)=\overrightarrow{\Omega}_{t,T}G \overleftarrow{U}_{T,t}+
	\int_t^T\overrightarrow{\Omega}_{t,r}(Q_{12}(r)-P(r)Q_1(r))\overleftarrow{U}_{r,t}dr.
\]
By Proposition \ref{AuxIntergralEquationProposition1} the unique solution of
this equation (and hence of Eq. \eqref{AuxIntergralEquation3}) has the form
\eqref{AuxIntergralEquation3Solution}
\end{proof}

\textbf{2.2.} Let us prove some results on the representation of the solution of
the Riccati integral equation \eqref{IntegralRiccatiEq}.
\begin{proposition}\label{RiccatiEqSolutionRepresenation1}
Let $P\in C_s(\Interval;\Lin(\Banach_1,\Banach_2))$, and let an evolution family
$\overleftarrow{\Psi}_{t,s}$ in $\Lin(\Banach_1)$ be a solution of the equation
\begin{equation}\label{PhiFamilityDefinition}
	\overleftarrow{\Psi}_{t,s}=\overleftarrow{U}_{t,s}-
	\int_s^t\overleftarrow{\Psi}_{t,r}B(r)P(r)\overleftarrow{U}_{r,s}dr
	\quad 0\leq s\leq t\leq T.
\end{equation}
The operator function $P$ is a solution of the Riccati integral equation
\eqref{IntegralRiccatiEq} if and only if
\begin{equation}\label{MildSolutionRepresentation1}
	P(t)=\overrightarrow{U}_{t,T}G\overleftarrow{\Psi}_{T,t}+
	\int_t^T \overrightarrow{U}_{t,r}C(r)\overleftarrow{\Psi}_{r,t}dr
\end{equation}
for each $t\in\Interval$.
\end{proposition}
\begin{proof}
The proof readily follows from Proposition \ref{AuxIntergralEquationProposition1},
where one must set $Q_1(t)=B(t)P(t)$ and
$\overleftarrow{\Omega}_{t,s}=\overleftarrow{\Psi}_{t,s}$.
\end{proof}
\begin{proposition}\label{RiccatiEqSolutionRepresenation2}
Let $P\in C_s(\Interval; \Lin(\Banach_1,\Banach_2))$,
let a forward evolution famuily $\overleftarrow{\Psi}_{t,s}$ in $\Lin(\Banach_1)$
be a solution of Eq. \eqref{PhiFamilityDefinition}, and let a backward evolution
family $\overrightarrow{\Psi}_{s,t}$ in $\Lin(\Banach_2)$
be a solution of the equation
\begin{equation}\label{PsiFamilityDefinition}
	\overrightarrow{\Psi}_{t,s}=\overrightarrow{U}_{t,s}-
	\int_t^s\overrightarrow{U}_{t,r}P(r)B(r)\overrightarrow{\Psi}_{r,s}dr
	\quad 0\leq t\leq s\leq T.
\end{equation}
The operator function $P$ is a solution of the Riccati
equation \eqref{IntegralRiccatiEq} if and only if the relation
\begin{equation}\label{MildSolutionRepresentation2}
	P(t)=\overrightarrow{\Psi}_{t,T}G\overleftarrow{\Psi}_{T,t}+
	\int_t^T\overrightarrow{\Psi}_{t,r}(C(r)+P(r)B(r)P(r))\overleftarrow{\Psi}_{r,t}dr
\end{equation}
holds for each $t\in\Interval$.
\end{proposition}
\begin{proof}
The proof readily follows from Proposition \ref{AuxIntergralEquationProposition3},
where one must set $\overrightarrow{\Omega}_{t,s}=\overrightarrow{\Psi}_{t,s}$,
$\overleftarrow{\Omega}_{t,s}=\overleftarrow{\Psi}_{t,s}$,
$Q_1(t)=B(t)P(t)$, $Q_2(t)=P(t)B(t)$ and $Q_{12}(t)=C(t)+P(t)B(t)P(t)$.
\end{proof}

\textbf{2.3.} To prove the uniqueness of the solution of the Riccati integral
equation, we need a generalization of the following result
\cite[Sec. IV, Lemma 2.2]{RepresentationAndControl}
\begin{proposition}\label{ContractionProposition}
Let
\[
	Q_0\in C_u([a,b];\Lin(\Banach_1,\Banach_2)),\quad
	B\in C_u([a,b];\Lin(\Banach_2,\Banach_1)),
\]
and let operator functions $\{Q_{s,t}^{(1,2)}\}_{a\leq t\leq s\leq b}$
ranging in  $\Lin(\Banach_{1,2})$ be strongly continuous separately in $t$ and
$s$ and uniformly bounded, $\|Q_{s,t}^{(1,2)}\|_{\Lin(\Banach_{1,2})}\leq M_{1,2}$.
Further, let $\rho$ be a number satisfying the inequalities
\begin{equation}\label{ConditionsOnRho}
	\|Q_0\|_u+\rho^2(b-a)M_1M_2\|B\|_u\leq\rho,\quad
	2\rho (b-a)M_1M_2\|B\|_u<1
\end{equation}
Then the mapping $\Gamma$ acting on the space $C_u([a,b];\Lin(\Banach_1,\Banach_2))$
by the rule
\[
	\Gamma(P)(t)=Q_0(t)-
	\int_t^b Q^{(2)}_{r,t}P(r)B(r)P(r)Q^{(1)}_{r,t}dr
\]
is a contraction on the ball
\[
	B_\rho=\{P\in C_u([a,b];\Lin(\Banach_1,\Banach_2)):
	\|P\|_u\leq \rho \}.
\]
\end{proposition}
\begin{proof}
Let us show that $\Gamma(B_\rho)\subseteq B_\rho$. Let $P\in B_\rho$.
For each $x\in\Banach_1$, we have
\[
	\|Q^{(2)}_{r,t}P(r)B(r)P(r)Q^{(1)}_{r,t}x\|_{\Banach_2}\leq
	\rho^2M_1M_2\|B\|_u\|x\|_{\Banach_1}.
\]
and hence
\begin{multline*}
	\|\Gamma(P)(t)x\|_{\Banach_2}=
	\left\|Q_0(t)-\int_t^bQ^{(2)}_{r,t}P(r)B(r)P(r)Q^{(1)}_{r,t}xdr\right\|_{\Banach_2}\leq\\
	\|Q_0(t)x\|_{\Banach_2}+
	\int_t^b\|Q^{(2)}_{r,t}P(r)B(r)P(r)Q^{(1)}_{r,t}x\|_{\Banach_2}dr\leq\\
	(\|Q_0\|_u+\rho^2M_1M_2(b-t)\|B\|_u)\|x\|_{\Banach_1}.
\end{multline*}
Since $x$ and $t\in[a,b]$ are arbitrary, we take into account the first
inequality in \eqref{ConditionsOnRho} and obtain
\[
	\|\Gamma(P)\|_u\leq
	\|Q_0\|_u+\rho^2(b-a)M_1M_2\|B\|_u\leq\rho,
\]
i.e., $\Gamma(P)\in B_\rho$. Further, let $P_1,P_2\in B_\rho$.
For each $x\in\Banach_1$,
\begin{multline*}
	\|Q^{(2)}_{r,t}\{P_2(r)B(r)P_2(r)-P_1(r)B(r)P_1(r)\}Q^{(1)}_{r,t}x\|_{\Banach_2}=\\
	M_1M_2\|(P_2(r)-P_1(r))B(r)P_2(r)+P_1(r)B(r)(P_2(r)-P_1(r))\| \|x\|_{\Banach_1}\leq\\
	M_1M_2\|P_2(r)-P_1(r)\|\|B(r)\|\{\|P_1(r)\|+\|P_2(r)\| \}
	\|x\|_{\Banach_1}\leq\\
	2\rho M_1M_2\|P_2-P_1\|_u\|B\|_u\|x\|_{\Banach_1}.
\end{multline*}
Then
\begin{multline*}
	\|\Gamma(P_1)(t)x-\Gamma(P_2)(t)x\|_{\Banach_2}=\\
	\left\|\int_t^bQ^{(2)}_{r,t}\{P_2(r)B(r)P_2(r)-P_1(r)B(r)P_1(r)\}Q^{(1)}_{r,t}xdr\right\|_{\Banach_2}\leq\\
	\int_t^b\left\|Q^{(2)}_{r,t}\{P_2(r)B(r)P_2(r)-P_1(r)B(r)P_1(r)\}Q^{(1)}_{r,t}x\right\|_{\Banach_2}dr\leq\\
	2\rho (b-t)M_1M_2\|P_2-P_1\|_u\|B\|_u\|x\|_{\Banach_1}.
\end{multline*}
or, finally,
\[
	\|\Gamma(P_1)-\Gamma(P_2)\|_u\leq
	\{2\rho (b-a)M_1M_2\|B\|_u\}\|P_1-P_2\|_u.
\]
Since $2\rho (b-a)M_1M_2\|B\|_u<1$ by the second inequality in
\eqref{ConditionsOnRho} it follows that the mapping $\Gamma$
in the space $C_u([a,b];\Lin(\Banach_1,\Banach_2))$
is a contraction mapping of the ball $B_\rho$ into itself.
The proof of the proposition is complete.
\end{proof}
\begin{corolario}\label{Corolario2}
Under the assumptions of Proposition \ref{ContractionProposition},
the equation
\[
	P(t)=Q_0(t)-\int_t^b Q^{(2)}_{r,t}P(r)B(r)P(r)Q^{(1)}_{r,t}dr
\]
has a unique solution $P\in C_u([a,b];\Lin(\Banach_1,\Banach_2))$ with
$\|P\|_u\leq\rho$.
\end{corolario}
\begin{proposition}\label{PropositionOnUniqueSolvability}
Let numbers  $r_B, r_C$ and $r_G$
satisfy the inequality
\[
	4(b-a)M_1^2M_2^2(r_G+(b-a)r_C)r_B<1
\]
and let operator functions $Q^{(1,2)}_{s,t}$ satisfy the assumptions of
Proposition \ref{ContractionProposition}.
Then the equation
\begin{equation}\label{AuxInteralEqiationOn_ab}
	P(t)=Q^{(2)}_{b,t}GQ^{(1)}_{b,t}+
	\int_t^b Q^{(2)}_{r,t}\{C(r)-P(r)B(r)P(r)\}Q^{(1)}_{r,t}dr
\end{equation}
has a unique solution $P\in C_u([a,b];\Lin(\Banach_1,\Banach_2))$
for arbitrary $G\in\Lin(\Banach_1,\Banach_2)$,
$C\in C_u([a,b];\Lin(\Banach_1,\Banach_2))$ and
$B\in C_u([a,b];\Lin(\Banach_2,\Banach_1))$ such that
\[
	\|G\|_{\Lin(\Banach_1,\Banach_2)}\leq r_G\quad
	\|C\|_u\leq r_C\quad \|B\|_u\leq r_B,
\]
and one has $\|P\|_u\leq 2M_1M_2(r_G+(b-a)r_C)$.
\end{proposition}
\begin{proof}
Set
\[
	Q_0(t)=Q^{(2)}_{b,t}GQ^{(1)}_{b,t}+\int_t^bQ^{(2)}_{r,t}C(r)Q^{(1)}_{r,t}dr
\]
Then Eq. \eqref{AuxInteralEqiationOn_ab} can be written in the form
$P=\Gamma(P)$, where the mapping $\Gamma(\cdot)$ is defined in
Proposition \ref{ContractionProposition}. Let $\rho=2M_1M_2(r_G+(b-a)r_C)$.
By virtue of the assumptions in the proposition to be proved, we have
\[
	2\rho(b-a)M_1M_2\|B\|_u\leq
	4M_1^2M_2^2(b-a)(r_G+(b-a)r_C)r_B<1.
\]
Further, since
\[
	\|Q_0\|_u\leq M_1M_2(\|G\|_{\Lin(\Banach_1,\Banach_2)}+(b-a)\|C\|_u)\leq
	M_1M_2(r_C+(b-a)r_C),
\]
we have
\begin{multline*}
	\|Q_0\|_u+\rho^2(b-a)M_1M_1\|B\|_u\leq\\
	M_1M_2(\|G\|_{\Lin(\Banach_1,\Banach_2)}+(b-a)\|C\|_u)+\rho^2(b-a)M_1M_1\|B\|_u<\\
	M_1M_2(r_G+(b-a)r_C)+\rho/2=\rho.
\end{multline*}
Hence inequalities \eqref{ConditionsOnRho} are satisfied for the number $\rho$.
Now the unique solvability of Eq. \eqref{AuxInteralEqiationOn_ab} follows
from Proposition \ref{ContractionProposition} and Corollary \ref{Corolario2}.
The proof of the proposition is complete.
\end{proof}

\section{Main result}

Let $\Banach$ be a reflexive Banach space. The duality between the spaces
$\Banach$ and $\Banach^*$ will be denoted by $\langle y,x\rangle$, where
$y\in\Banach^*$ and $x\in\Banach$.

Let $A_1\in\Lin(\Banach,\Banach^*)$. Then the adjoint operator
$A_1^*\in\Lin(\Banach^{**},\Banach^*)$. Using the canonical isomorphism
between the spaces $\Banach^{**}$ and $\Banach$ we can treat the adjoint
operator as $A_1^*\in\Lin(\Banach,\Banach^*)$.
A straightforward verification shows that
\[
	\langle A_1x_1,x_2\rangle=\overline{\langle A_1^*x_2,x_1\rangle}\quad
	\forall x_1,x_2\in\Banach.
\]
\begin{definicion}
An operator $A_1\in\Lin(\Banach,\Banach^*)$ is said to be \emph{self-adjoint}
if $A_1=A_1^*$. This is equivalent to the condition that
$\langle A_1x,x\rangle\in\R$ for all $x\in\Banach$.

We say that a self-adjoint operator $A_1\in\Lin(\Banach,\Banach^*)$ is
\emph{nonnegative} and write $A_1\geq0$ if $\langle A_1x,x\rangle\geq0$
for all $x\in\Banach$.
\end{definicion}
In a similar way, if $A_2\in\Lin(\Banach^*,\Banach)$, then, identifying
the spaces $\Banach^{**}$ and $\Banach$, we assume that
$A_2^*\in\Lin(\Banach^*,\Banach)$; a straightforward verification shows that
\[
	\langle x_1,A_2x_2\rangle=\overline{\langle x_2, A_2^* x_1\rangle}\quad
	\forall x_1,x_2\in\Banach^*.
\]
\begin{definicion}
An operator $A_2\in\Lin(\Banach^*,\Banach)$ is said to be \emph{elf-adjoint},
if $A_2=A_2^*$. This is equivalent to the condition that
$\langle x,A_2x\rangle\in\R$ for all s$x\in\Banach^*$.

We say that a self-adjoint operator $A_2\in\Lin(\Banach^*,\Banach)$ is 
\emph{nonnegative} and write $A_2\geq0$ if $\langle x,A_2x\rangle\geq0$ for all
$x\in\Banach^*$.
\end{definicion}

Let us state the main result about the unique solvability of the Riccati
integral equation.
\begin{teorema}\label{MainResult}
Let $\Banach$ be a reflexive Banach space, and let the following conditions
be satisfied:
\begin{enumerate}
	\item $\{\overleftarrow{U}_{t,s}\}_{0\leq s\leq t\leq T}$ is a strongly
	continuous uniformly bounded forward evolution family in $\Lin(\Banach)$;
	\item $\overrightarrow{U}_{t,s}=\left(\overleftarrow{U}_{s,t}\right)^*$ is
	a backward evolution family. (Since $\Banach$ is s reflexive,
	it follows that this family is strongly continuous and uniformly bounded in
	$\Lin(\Banach^*)$);
	\item The operator functions $C$ and $B$ satisfy the inclusions
	\[
		C\in C_s(\Interval;\Lin(\Banach,\Banach^*))\quad \text{and}\quad
		B\in C_s(\Interval;\Lin(\Banach^*,\Banach));
	\]
	\item $C(t)=C^*(t)\geq0$ and $B(t)=B^*(t)\geq0$ for all $t\in\Interval$.
\end{enumerate}
Then for an arbitrary self-adjoint nonnegative operator
$G\in\Lin(\Banach,\Banach^*)$
the Riccati integral equation \eqref{IntegralRiccatiEq} has a unique
solution $P\in C_s(\Interval;\Lin(\Banach,\Banach^*))$ and
$P(t)=P^*(t)\geq0$ for all $t\in\Interval$.
\end{teorema}
\begin{proof}
Following \cite{RepresentationAndControl, CurtainPritchard, PritchardSalamon},
consider the sequence
\[
	\{P_n\}_{n=0}^{+\infty}\subset C_s(\Interval;\Lin(\Banach,\Banach^*))
\]
of operator functions defined recursively as follows.
Set $P_0(t)\equiv0$ and define $P_{n+1}(t)$ as the solution of the equation
\begin{multline}\label{PnIntegralEquation}
	P_{n+1}(t)=\overrightarrow{U}_{t,T}G\overleftarrow{U}_{T,t}+
	\int_t^T\overrightarrow{U}_{t,r}\{C(r)+\\ P_n(r)B(r)P_n(r)-
	P_{n+1}(r)B(r)P_n(r)-P_n(r)B(r)P_{n+1}(r)\}\overleftarrow{U}_{r,t}dr
\end{multline}
for each $n\geq0$. By $\overleftarrow{\Psi}^{(n)}_{t,s}\in\Lin(\Banach)$ and
$\overrightarrow{\Psi}^{(n)}_{s,t}\in\Lin(\Banach^*)$ we denote the solutions
of Eqs. \eqref{PhiFamilityDefinition} and \eqref{PsiFamilityDefinition},
respectively, with $P=P_{n}$.
By Proposition \ref{AuxIntergralEquationProposition3},
the solution of Eq. \eqref{PnIntegralEquation} has the form
\begin{equation}\label{PnFormula}
	P_{n+1}(t)=\overrightarrow{\Psi}^{(n)}_{t,T}G\overleftarrow{\Psi}^{(n)}_{T,t}+
	\int_t^T\overrightarrow{\Psi}^{(n)}_{t,r}\{C(r)+P_{n}(r)B(r)P_{n}(r)\}\overleftarrow{\Psi}^{(n)}_{r,t}dr
\end{equation}
If $P^*_{n}(t)=P_{n}(t)$ for all $t\in\Interval$, then it follows from
Eqs. \eqref{PhiFamilityDefinition} and
\eqref{PsiFamilityDefinition}, the self-adjointness of the operator function
$B(t)$ and the condition
$\overrightarrow{U}_{t,s}=\left(\overleftarrow{U}_{s,t}\right)^*$ that
$\overrightarrow{\Psi}^{(n)}_{t,s}=\left(\overleftarrow{\Psi}^{(n)}_{s,t}\right)^*$.

In view of this equality, it readily follows from Eq. \eqref{PnFormula}
that the self-adjointness of the operator function $P_n$
implies the self-adjointness of the operator function $P_{n+1}$.
Since $P_0=0$, we see that $\{P_n\}_{n=0}^{+\infty}$ is a sequence of
self-adjoint operator functions. Further, a straightforward verification
for Eq. \eqref{PnFormula} shows that the nonnegativity of the operators
($t\in\Interval$)
\[
	G,C(t),B(t)\geq0
\]
implies the nonnegativity of the operator function $P_{n+1}$

Let us show that $P_{n+1}(t)\leq P_{n}(t)$ (i.e. $P_n(t)-P_{n+1}(t)\geq0$)
for all $t\in[0,T]$ and $n\in\N$. Let $Q_{n+1}=P_{n+1}-P_{n}$.
Let us subtract Eq. \eqref{PnIntegralEquation} written for
$P_{n}$ from from the same equation for $P_{n+1}$. After obvious
transformations, we find that the operator function
$Q_{n+1}$ satisfies the equation
\begin{multline*}
	Q_{n+1}(t)=-	\int_t^T\overrightarrow{U}_{t,r}\{Q_n(r)B(r)Q_n(r)+\\
	Q_{n+1}(r)B(r)P_n(r)+P_n(r)B(r)Q_{n+1}(r)\}\overleftarrow{U}_{r,t}dr
\end{multline*}
By Proposition \ref{AuxIntergralEquationProposition3},
the unique solution of this equation has the form
\[
	Q_{n+1}(t)=-\int_t^T \overrightarrow{\Psi}^{(n)}_{t,r}Q_n(r)B(r)Q_n(r)\overleftarrow{\Psi}^{(n)}_{r,t}dr.
\]
Since $Q_n^*(t)=Q_n(t)$ and $B(t)\geq0$, we have
$Q_{n+1}(t)=P_{n+1}(t)-P_n(t)\leq0$ for all $t\in\Interval$.

Thus, $\{P_n(t)\}_{n=1}^{+\infty}$ is a monotone nonincreasing sequence of
nonnegative self-adjoint operators in $\Lin(X,X^*)$ for each $t\in\Interval$.
Then \cite[Th. 4]{Koshkin} for each
$t\in\Interval$ there exists a strong limit
\[
	\slim_{n\to+\infty}P_n(t)=P(t)\in\Lin(X,X^*),
\]
and $P^*(t)=P(t)\geq0$. Moreover, since \cite[Th. 4]{Koshkin}
\[
	\|P_n(t)\|_{\Lin(\Banach,\Banach^*)}=
	\sup_{x\in\Banach,\|x\|=1}\langle P_n(t)x,x\rangle,
\]
we see that the numerical sequence $\|P_n(t)\|_{\Lin(\Banach,\Banach^*)}$
is monotone decreasing for each $t\in\Interval$. 

Since $\{P_n\}_{n=0}^{+\infty}$ are uniformly bounded, then there exist  strong limits
\begin{multline*}
	\slim_{n\to+\infty}\overrightarrow{U}_{t,r}P_{n}(r)B(r)P_{n}(r)\overleftarrow{U}_{r,t}=
	\slim_{n\to+\infty}\overrightarrow{U}_{t,r}P_{n+1}(r)B(r)P_{n}(r)\overleftarrow{U}_{r,t}=\\
	\slim_{n\to+\infty}\overrightarrow{U}_{t,r}P_{n}(r)B(r)P_{n+1}(r)\overleftarrow{U}_{r,t}=
	\overrightarrow{U}_{t,r}P(r)B(r)P(r)\overleftarrow{U}_{r,t}
\end{multline*}
for each $0\leq t\leq r\leq T$.

Let us fix an $x\in\Banach$ and write ($t\leq r$)
\begin{gather*}
	f_n(t)=\langle P_n(t)x,x\rangle,\quad
	f(t)=\langle P(t)x,x\rangle\\
	g(t,r)=-\langle \overrightarrow{U}_{t,r}P(r)B(r)P(r)\overleftarrow{U}_{r,t}x,
	x\rangle\\
	h(t)=\langle \overrightarrow{U}_{t,T}G\overleftarrow{U}_{T,t}
	x\rangle+\int_t^T\langle\overrightarrow{U}_{t,r}C(r)\overleftarrow{U}_{r,t}x,
	x\rangle dr
\end{gather*}
and
\begin{multline*}
	g_n(t,r)=\langle\overrightarrow{U}_{t,r}P_n(r)B(r)P_n(r)\overleftarrow{U}_{r,t}x,x\rangle-\\
	\langle\overrightarrow{U}_{t,r}\{P_{n+1}(r)B(r)P_n(r)+P_n(r)B(r)P_{n+1}(r)]\}\overleftarrow{U}_{r,t}x,
	x\rangle.
\end{multline*}
By construction, 
\begin{itemize}
	\item $f_n\in C(\Interval)$, $f_n(t)\geq0$, $f_{n+1}(t)\leq f_n(t)$
	and the sequence $f_n(t)$ converges to $f(t)$ pointwise on $\Interval$;
	\item $g_n(t,\cdot)\in C([t,T])$ are uniformly bounded and 
	the sequence $g_n(t,\cdot)$ converges to
	$g(t,\cdot)$ pointwise on $[t,T]$ for each $t\in\Interval$;
	\item $g_n(\cdot,r),g(\cdot,r)\in C([0,r])$ for each $r\in\Interval$;
	\item $h\in C(\Interval)$ and Eq. \eqref{PnIntegralEquation} implies
	the equality ($t\in\Interval$)
	\begin{equation}\label{fnIntegralEquality}
		f_n(t)=h(t)+\int_t^T g_n(t,r)dr.
	\end{equation}
\end{itemize}
Thus $f$ and $g(t,\cdot)$ are bounded and measurable. Passing to the limit as
$n\to+\infty$ in \eqref{fnIntegralEquality} we obtain ($t\in\Interval$)
\[
	f(t)=h(t)+\int_t^T g(t,r)dr.
\]
Since $g(\cdot,r)$ is continuous for each $r\in\Interval$ and uniformely bounded, then 
$\int_t^T g(t,r)dr\in C(\Interval)$. Thus $f\in C(\Interval)$.

By virtue of Dini's theorem the functional sequence $\{f_n\}_{n=1}^{+\infty}$ 
converges to $f$ uniformly on $\Interval$. Further, for any $n,m\in\N$
the inequality $P_{n+m}(t)\leq P_n(t)$ implies the estimate
\cite[Th. 4]{Koshkin}
\begin{multline*}
	\|P_n(t)x-P_{n+m}(t)x\|^2_{\Banach}\leq\\
	\|P_n(t)-P_{n+m}(t)\|_{\Lin(\Banach,\Banach^*)}
	\langle P_n(t)x-P_{n+m}(t)x,x\rangle\leq\\
	\mathrm{const}(f_n(t)-f_{n+m}(t)),
\end{multline*}
for all $t\in\Interval$, where the constant does not depend on $t$.
According to the Cauchy convergence test, we find that the uniform
convergence on $\Interval$ of the sequence $\{f_n\}_{n=1}^{+\infty}$
implies the uniform convergence of the sequence $P_nx$ to $Px$
on $\Interval$ and the inclusion $Px\in C(\Interval,\Banach^*)$.
Since $x$ is arbitrary, we conclude that the sequence $\{P_n\}_{n=1}^{+\infty}$
converges in the space $C_s(\Interval;\Lin(\Banach,\Banach^*))$
and one has the inclusion $P\in C_s(\Interval;\Lin(\Banach,\Banach^*))$.

Let $\overleftarrow{\Psi}_{t,s}\in\Lin(\Banach)$ and $\overrightarrow{\Psi}_{s,t}\in\Lin(\Banach^*)$ be solutions of Eqs.
\eqref{PhiFamilityDefinition} and \eqref{PsiFamilityDefinition},
respectively. Let us show that $P$ is a solution of the Riccati integral
equation \eqref{IntegralRiccatiEq}. Since
\[
	B\in C_s(\Interval;\Lin(\Banach^*,\Banach)),\quad
	P_n,P\in C_s(\Interval;\Lin(\Banach,\Banach^*)),
\]
it follows that the sequence $BP_n$ converges to $BP$ in the space
$C_s(\Interval;\Lin(\Banach))$.
Applying Proposition \ref{ContinuousDepOfPhiEvolutionFamily} with $Q_n=BP_n$,
we find that there exists a limit
$\slim_{n\to+\infty}\overleftarrow{\Psi}^{(n)}_{r,t}=\overleftarrow{\Psi}_{r,t}$
uniformly with respect to $r\in[t,T]$. It can be shown in a similar way that
there exists a limit
$\slim_{n\to+\infty}\overrightarrow{\Psi}^{(n)}_{t,r}=\overrightarrow{\Psi}_{t,r}$
uniformly with respect to $r\in[t,T]$.

It follows from the uniform convergence of the sequences
$P_n$, $\overleftarrow{\Psi}^{(n)}_{r,t}$ and
$\overrightarrow{\Psi}^{(n)}_{t,r}$ to $P$, $\overleftarrow{\Psi}_{r,t}$ and
$\overrightarrow{\Psi}_{t,r}$, respectively, that for each $r\in[t,T]$
the following limit exists uniformly with respect to $r\in[t,T]$
\[
	\slim_{n\to+\infty}\overrightarrow{\Psi}^{(n)}_{t,r}\{C(r)+P_n(r)B(r)P_n(r)\}
	\overleftarrow{\Psi}^{(n)}_{r,t}=\overrightarrow{\Psi}_{t,r}
	\{C(r)+P(r)B(r)P(r)\}\overleftarrow{\Psi}_{r,t}
\]
Passing to the limit as $n\to+\infty$ in \eqref{PnFormula} we see that
the operator function $P$ satisfies Eq. \eqref{MildSolutionRepresentation2}
in which the evolution families $\overleftarrow{\Psi}_{r,t}$ and
$\overrightarrow{\Psi}_{t,r}$ are determined by
Eqs. \eqref{PhiFamilityDefinition} and \eqref{PsiFamilityDefinition},
respectively. It follows from Proposition \ref{RiccatiEqSolutionRepresenation2}
that the operator function $P\in C_s(\Interval;\Lin(\Banach,\Banach^*))$
is a solution of the Riccati integral equation \eqref{IntegralRiccatiEq}.

Let us prove the uniqueness of the solution of this equation. Set
\[
	r=\sup_{t\in\Interval}\|P(t)\|_{\Lin(\Banach,\Banach^*)}\quad
	r_C=\sup_{t\in\Interval}\|C(t)\|_{\Lin(\Banach,\Banach^*)}\quad
	r_B=\sup_{t\in\Interval}\|B(t)\|_{\Lin(\Banach^*,\Banach)}
\]
Since $P(T)=G$, we have $r\geq\|G\|_{\Lin(\Banach,\Banach^*)}$. \
Let the following inequality be satisfied for $\delta>0$
\[
	4\delta M_1^2M_2^2(r+\delta r_C)r_B<1.
\]
If $Q$ is another solution of Eq. \eqref{IntegralRiccatiEq}, then, applying
Proposition \ref{PropositionOnUniqueSolvability} to the interval $[T-\delta,T]$
(setting $Q^{(1)}_{r,t}=\overleftarrow{U}_{r,t}$ and
$Q^{(2)}_{r,t}=\overrightarrow{U}_{t,r}$), we obtain
$P(t)=Q(t)$ for all $t\in[T-\delta,T]$. Further,
\[
	P(t)=\overrightarrow{U}_{t,T-\delta}P(T-\delta)\overleftarrow{U}_{T-\delta,t}+
	\int_t^{T-\delta} \overrightarrow{U}_{t,r}\{C(r)-P(r)B(r)P(r)\}\overleftarrow{U}_{r,t}dr.
\]
for all $0\leq t\leq T-\delta$. Applying Proposition
\ref{PropositionOnUniqueSolvability} to the interval $[T-2\delta,T-\delta]$,
we conclude that $P(t)=Q(t)$ for all $t\in[T-2\delta,T-\delta]$.
Continuing this process, we see that $P(t)=Q(t)$ for all $t\in\Interval$.
This proves the uniqueness of the solution of the Riccati integral equation
\eqref{IntegralRiccatiEq}. The proof of the theorem is complete.
\end{proof}

\end{document}